\documentclass{article}
\usepackage[journal=LEM,lang=british]{ems-journal-lem}

\usepackage[OT2,T1]{fontenc}
\DeclareSymbolFont{cyrletters}{OT2}{wncyr}{m}{n}
\DeclareMathSymbol{\sh}{\mathbin}{cyrletters}{"78}

\usepackage{verbatim}

\let\mr=\MR

\RequirePackage[initials,lite]{amsrefs}

\usepackage{mathscinet}

\let\MR=\mr

\usepackage{tikz-cd}

\numberwithin{equation}{section}

\newtheorem{theorem}{Theorem}
\newtheorem{proposition}{Proposition}
\newtheorem{corollary}{Corollary}
\newtheorem{lemma}{Lemma}
\newtheorem{definition}{Definition}

\newcommand{\F}{\mathbb{F}}
\newcommand{\R}{\mathbb{R}}
\newcommand{\Wedge}{\Lambda}
\newcommand{\<}{\langle}
\renewcommand{\>}{\rangle}

\DeclareMathOperator{\Hom}{Hom}
\DeclareMathOperator{\End}{End}
\DeclareMathOperator{\ad}{ad}

\newcommand{\Ainf}{$A_\infty$}
\newcommand{\Linf}{$L_\infty$}

\newcommand{\h}{\mathbf{h}}
\newcommand{\p}{\mathbf{p}}
\newcommand{\f}{\mathbf{f}}
\newcommand{\g}{\mathbf{g}}

\newcommand{\CA}{\mathcal{A}}
\newcommand{\CD}{\mathcal{D}}

\newcommand{\ddd}{\check{\vphantom{\partial}\partial}}
\newcommand{\hhh}{\check{\vphantom{\partial}\mathsf{h}}}
\newcommand{\fff}{\check{\vphantom{\partial}\mathsf{f}}}
\renewcommand{\ggg}{\check{\vphantom{\partial}\mathsf{g}}}

\newcommand{\dd}{\partial}
\newcommand{\hh}{\mathsf{h}}
\newcommand{\ff}{\mathsf{f}}
\renewcommand{\gg}{\mathsf{g}}

\DeclareMathOperator{\sgn}{sgn}
\DeclareMathOperator{\Cone}{Cone}

\title{Koszul duality and the Poincar\'e--Birkhoff--Witt theorem}

\begin{document}

\emsauthor{1}{
	\givenname{Ezra}
	\surname{Getzler}
	\mrid{210138}
	\orcid{0000-0002-5850-7723}}{Ezra~Getzler}

\Emsaffil{1}{
	\department{Department of Mathematics}
	\organisation{Northwestern University}
	\address{2033 Sheridan Road}
	\zip{60208}
	\city{Evanston}
	\country{USA}
	\affemail{getzler@northwestern.edu}
      }

\dedication{In memoriam Murray Gerstenhaber (1927-2024)}

\classification[13D03,16S37,17B35]{16E40}


\begin{abstract}
  Using a homotopy introduced by de Wilde and Lecomte and homological
  perturbation theory for \Ainf-algebras, we give an explicit proof
  that the universal enveloping algebra $UL$ of a differential graded
  Lie algebra $L$ is Koszul, via an explicit contracting homotopy from
  the cobar construction $\Omega CL$ of the Chevalley--Eilenberg chain
  coalgebra $CL$ of $L$ to $UL$. This may be viewed as an extension of
  the Poincar\'e--Birkhoff--Witt Theorem to \Linf-algebras.
\end{abstract}

\maketitle

\thispagestyle{empty}

If $L$ is a differential graded (dg) Lie algebra, there is a
codifferential on the exterior coalgebra $\Wedge L$ defined using the
Lie bracket and differential of $L$. With this differential,
$\Wedge L$ becomes a cocommutative dg coalgebra $CL$, called the
Chevalley--Eilenberg coalgebra of $L$.

There is a quasi-isomorphism of cocommutative dg bialgebras
$f:\Omega CL\to UL$ from the cobar construction $\Omega CL$ of $CL$ to the
universal enveloping algebra $UL$, split by a morphism of
cocommutative dg coalgebras $g:UL\to\Omega CL$. Applying the functor of
primitives $P$ to $f$, we obtain a quasi-isomorphism of dg Lie
algebras $Pf:P\Omega CL\to L$: this functorial resolution of $L$ was
introduced by Quillen \cite{Quillen}.

More generally, if $L$ is an \Linf-algebra, there is a codifferential
on $\Wedge L$ defined using the higher brackets of $L$, that makes
$\Wedge L$ into a cocommutative dg coalgebra $CL$. A natural choice
for the universal enveloping algebra of $L$ is $\Omega CL$ (Hinich and
Schechtman \cite{HS}). This is a cocommutative dg bialgebra, and as we
saw above, in the special case that $L$ is a dg Lie algebra it is
quasi-isomorphic to the universal enveloping algebra $UL$.

It is an interesting problem to exhibit this quasi-isomorphism by
means of an explicit contracting homotopy. In the abelian case, this
becomes the problem of finding an explicit contracting homotopy from
$\Omega\Wedge V$, where $V$ is a cochain complex, to $SV$, the dg symmetric
algebra generated by $V$. A contracting homotopy from $\Omega CL$ to
$UL$ is then obtained by homological perturbation theory (Baranovsky
\cite{Baranovsky}). The resulting identification of the complexes
underlying $SL$ and $UL$ may be viewed as an extension of the
Poincar\'e--Birkhoff--Witt theorem to \Linf-algebras.

Baranovsky demonstrated the existence of a contracting homotopy for
$\Omega\Wedge V$, with\-out giving an explicit formula for it. In fact,
such an explicit homotopy may be extracted from the work of de Wilde
and Lecomte \cite{dWL}. They construct a homotopy in the dual
situation, contracting the bar construction $BSV$ of the symmetric
algebra $SV$ to its Koszul dual, the exterior coalgebra $\Wedge V$
(though they restrict attention to the case that $V$ is a vector
space). Dualizing, we obtain a contracting homotopy from the cobar
construction $\Omega\Wedge V$ to the symmetric algebra $SV$. This homotopy
is natural: it is invariant under automorphisms of $V$. De~Wilde and
Lecomte construct this homotopy by a recursive procedure: one of the
new results in our paper is an explicit formula \eqref{explicit} for
the resulting homotopy.

Another approach to the construction of a contracting homotopy is due
to Halbout \cite{Halbout}. His homotopy extends to more general
function algebras, such as real (or complex) analytic functions. We
will see that the homotopy of de~Wilde and Lecomte extends to
Halbout's setting as well.

After the first version of this work was completed, an alternative
contracting homotopy was discovered by Dippell et al.\ \cite{DESW}. An
advantage of their homotopy over the one that we use is that its
origin is conceptually clearer: it is obtained by combining
contracting homotopies of the bar construction and the Koszul complex,
in an abstraction of Van~Est's method (see Meinrenken and Salazar
\cite{MS}). For completeness, we give a brief review of their
construction in the appendix.

The articles of de~Wilde and Lecomte \cite{dWL}, Halbout
\cite{Halbout} and Dippel et al.~\cite{DESW} consider the more general
setting of Hochschild complexes of chains and cochains with values in
modules over the symmetric algebra $SV$. We have not pursued these
applications in this article, but the respective generalizations are
easily guessed from the homotopy formulas that we present here.

The original motivation of this paper was to make the \Ainf-morphism
in Lemma~19 of Tsygan \cite{Tsygan} explicit: this is the last step in
his construction of a Gauss--Manin connection on periodic cyclic
homology of a deformation of \Ainf-algebras at the chain level, and
the only step for which no explicit formula is stated. We return to
this application in the last section of this paper.

\clearpage

\subsection*{Summary of Paper} \mbox{}

\S1 reviews the definitions of the symmetric algebra $SV$ and
exterior coalgebra $\Wedge V$.

\S2 discusses operations on the bar construction $BA$ of an
\Ainf-algebra $A$, parametrized by the bar construction $BG(A)$ of its
Gerstenhaber algebra $G(A)=C^*(A,A)$.

\S3 reviews homological perturbation theory for complexes, and using
the tensor trick of \cites{GLS,HK} for \Ainf-algebras.

\S4 studies the contracting homotopy from $BSV$ to $\Wedge V$.

\S5 studies the dual contracting homotopy from $\Omega \Wedge V$ to
$SV$, and by homological perturbation theory, from $\Omega CL$ to $UL$ for
an \Linf-algebra $L$. The same method is used in \S6 to make the
formula for the \Ainf-morphism in Lemma 19 of \cite{Tsygan} explicit.

In an appendix, we present the alternative contracting homotopy of
\cite{DESW}.

\section{The symmetric algebra and exterior coalgebra of a cochain complex}

We work with complexes over a field $\F$ of characteristic zero,
graded cohomologically, so that the differential has degree $1$. We
use the notation $sV$ for the suspension of a complex:
$(sV)^i=V^{i+1}$. All tensor products are over $\F$ unless otherwise
indicated.

Koszul duality is streamlined by working in the category of non-unital
dg algebras and \Ainf-algebras, and non-counital coalgebras. For this
reason, we define the symmetric algebra $SV$ of a complex $V$ to be
\begin{equation*}
  SV = \bigoplus_{k=1}^\infty S^kV ,
\end{equation*}
and the exterior coalgebra $\Wedge V$ of $V$ to be
\begin{equation*}
  \Wedge V = \bigoplus_{k=1}^\infty \Wedge_kV ,
\end{equation*}
where $\Wedge_kV\cong S^ksV$.

To a non-counital dg coalgebra $C$, we may associate a counital
coalgebra $C_+=C\oplus\F$, with counit $\epsilon:C_+\to\F$ equal to projection on
the summand $\F$, and codifferential extended to $C_+$ by vanishing on
$\F\subset C_+$.The coproduct $\Delta_+$ on $C_+$ is given by the formula
\begin{equation*}
  \Delta_+ a = a\otimes 1 + \Delta a + 1 \otimes a - \epsilon(a) \, 1 \otimes 1 .
\end{equation*}
The coproduct $\Delta$ on the non-counital coalgebra is sometimes called
the reduced coproduct. We make use of Sweedler's notation,
abbreviating the coproduct of a coalgebra
\begin{equation*}
  \Delta a = \sum_i a^{(1)}_i \otimes a^{(2)}_i
\end{equation*}
to
\begin{equation*}
  \Delta a = a^{(1)} \otimes a^{(2)} .
\end{equation*}
An element $x\in C$ of a non-counital dg coalgebra is primitive if its
coproduct $\Delta x \in C\otimes C$ vanishes: the set $PC\subset C$ of primitive elements
is a subcomplex of $C$.

Likewise, if $A$ is a dg algebra, let $A^+=A\oplus\F$, with product
\begin{equation*}
  (a+\alpha1)(b+\beta1)=(ab+\alpha b+\beta a)+(\alpha\beta)1 ,
\end{equation*}
and differential $\delta(a+\alpha1)=\delta a$. In particular,
\begin{equation*}
  SV^+ = \bigoplus_{k=0}^\infty S^kV ,
\end{equation*}
and
\begin{equation*}
  \Wedge V_+ = \bigoplus_{k=0}^\infty \Wedge_kV .
\end{equation*}

Working with non-unital algebras and non-counital coalgebras has the
disadvantage of rendering the definition of a bialgebra less
intuitive. The compatibility between the product and coproduct may be
written
\begin{equation*}
  \Delta(ab) = ( a \otimes 1 + \Delta a + 1 \otimes a )  ( b \otimes 1 + \Delta b + 1 \otimes b ) - ab \otimes 1 -
  1 \otimes ab .
\end{equation*}
In Sweedler notation, this becomes
\begin{multline*}
  (ab)_{(1)} \otimes (ab)_{(2)} - a_{(1)}b_{(1)} \otimes a_{(2)}b_{(2)} \\
  = a \otimes b + a_{(1)}b \otimes a_{(2)} + a_{(1)} \otimes a_{(2)}b + ab_{(1)} \otimes
  b_{(2)} + b_{(1)} \otimes ab_{(2)} + b \otimes a ,
\end{multline*}
or in the graded case,
\begin{multline*}
  (ab)_{(1)} \otimes (ab)_{(2)} - (-1)^{|a_{(2)}||b_{(1)}|} a_{(1)}b_{(1)} \otimes
  a_{(2)}b_{(2)} \\
  \begin{aligned}
    &= a \otimes b + (-1)^{|a_{(2)}||b|} a_{(1)}b \otimes a_{(2)} + a_{(1)} \otimes
    a_{(2)}b \\
    &\quad + ab_{(1)} \otimes b_{(2)} + (-1)^{|a||b_{(1)}|} b_{(1)} \otimes ab_{(2)} +
    (-1)^{|a||b|} b \otimes a .
  \end{aligned}
\end{multline*}

Both $SV$ and $\Wedge V$ are bialgebras: the coproduct of $SV$ is the
coshuffle coproduct, characterized by $PSV=V=S^1V$, and the product of
$\Wedge V$ is the wedge product
$a\wedge b:\Wedge_kV \otimes \Wedge_\ell V\to\Wedge_{k+\ell}V$.

A twisting cochain from a dg coalgebra $C$ to a dg algebra $A$ is a
morphism $\mu:C\to A$ of degree $1$ satisfying the Maurer--Cartan equation
$d\mu+\mu^2=0$, where $\mu^2$ is the composition $C\to
C^{\otimes2}\xrightarrow{\mu\otimes\mu} A^{\otimes2}\to A$.

A coderivation of a coalgebra $C$ is a map $\delta:C\to C$ such that
\begin{equation*}
  \Delta\delta = (\delta\otimes1+1\otimes\delta)\Delta .
\end{equation*}
A codifferential is a coderivation $\delta$ of degree $1$ such that
$\delta^2=0$.

An \Linf-structure on a complex $L$ is a codifferential $\delta$ on
$\Wedge L$, that is, a coderivation of degree $1$ whose square
vanishes. We have
\begin{equation*}
  \delta = d + \sum_{k=2}^\infty \delta_k ,
\end{equation*}
where $\delta_k$ is the component of $\delta$ that maps
$\Wedge_\ell L$ to $\Wedge_{\ell-k+1}L$. These coderivations are given by
the formulas
\begin{multline*}
  d \bigl( v_1 \wedge \cdots \wedge v_\ell \bigr) = \sum_{j=1}^\ell
  (-1)^{|v_1|+\cdots+|v_{j-1}|+j-1} v_1 \wedge \ldots \wedge dv_j \wedge \ldots \wedge v_\ell \\
  \shoveleft{\delta_k \bigl( v_1 \wedge \cdots \wedge v_\ell \bigr) = \frac{1}{\ell!}
    \binom{\ell}{k} \sum_{\pi\in S_\ell} \sgn(\pi) \, (-1)^{\sum_{i=1}^k(k-i)|v_i|}} \\
  [v_{\pi(1)},\ldots,v_{\pi(k)}] \wedge v_{\pi(k+1)} \wedge \ldots \wedge v_{\pi(\ell)} ,
\end{multline*}
where $[v_1,\ldots,v_k]$ is a graded antisymmetric $k$-linear bracket of
degree $2-k$, and $\sgn(\pi)$ is the sign associated to the action of
the permutation $\pi$ on $sv_1\otimes\cdots\otimes sv_\ell$:
\begin{equation}
  \label{sgn}
  \sgn(\pi) = (-1)^{\sum_{\{i<j\mid\pi(i)>\pi(j)\}} (|v_{\pi(i)}|+1)(|v_{\pi(j)}|+1)} .
\end{equation}
The coderivation $d$ may be viewed as the coderivation $\delta_1$
associated to the $1$-linear bracket $[v]=dv$.

The formula $\delta^2=0$ imposes quadratic equations among the brackets and
the differential $d$, which in the case where $[v_1,\ldots,v_k]$ vanishes
for $k>2$ become the usual axioms for a dg Lie algebra. This
codifferential makes $\Wedge L$ into a dg coalgebra denoted $CL$, the
(reduced) Chevalley--Eilenberg complex of chains of the \Linf-algebra
$L$. In the special case that $L$ is a Lie algebra, this recovers the
complex introduced by Chevalley and Eilenberg.

\newcommand{\Abar}{\bar{A}}
\newcommand{\Cbar}{\bar{C}}

\section{The bar construction of an \Ainf-algebra}

The bar construction of a complex $A$ is the graded vector space
\begin{equation*}
  BA = \bigoplus_{k=1}^\infty B_kA ,
\end{equation*}
where $B_kA = (sA)^{\otimes k}$. We denote the tensor product
\begin{equation*}
  sa_1 \otimes \cdots \otimes sa_k \in B_kA
\end{equation*}
by the bar notation $[a_1|\ldots|a_k]$ of Eilenberg and Maclane (from which
the construction derives its name). For $0\le j\le k$, let
\begin{equation*}
  \omega_j = |a_1|+\cdots+|a_j|-j ;
\end{equation*}
this is the degree of the element $[a_1|\ldots|a_j]$ of $B_jA$.

The bar construction is a dg coalgebra, with coproduct
$\Delta : BA \to BA \otimes BA$ and codifferential
$\delta : BA \to BA$ given by the formulas
\begin{align*}
  \Delta[a_1|\ldots|a_k] &= \sum_{0<j<k} [a_1|\ldots|a_j] \otimes [a_{j+1}|\ldots|a_k] , \\
  \delta[a_1|\ldots|a_k] &= \sum_{0\le j\le k} (-1)^{\omega_{j-1}} \, [a_1|\ldots|da_j|\ldots|a_k] ,
\end{align*}

The counital form of the bar construction is
\begin{equation*}
B_+A = BA \oplus \F = \bigoplus_{k=0}^\infty B_kA .
\end{equation*}
The counit $\epsilon:B_+A\to\F$ projects from $B_+A$ to $B_0A\cong\F$, with basis
vector $[~]$.

A Hochschild cochain $D\in C^*(A,A)$ is a map $D:B_+A\to A$. Denote the
component of $D$ in $\Hom(B_kA,A)$ by $D_{(k)}$. We denote the graded
vector space $C^*(A,A)$ of Hochschild cochains by $G(A)$.\footnote{The
  letter $G$ stands for Gerstenhaber \cite{Gerstenhaber}, who made the
  first close study of the algebraic properties of $G(A)$.}

There is a bijection between coderivations $\delta$ of $BA$ and Hochschild
cochains $D\in G(A)$ such that $D_{(0)}=0$, given by composition with
projection from $BA$ to $B_1A\cong sA$ followed by the degree one
isomorphism from $sA$ to $A$. We denote the coderivation corresponding
to a Hochschild cochain $D$ by $\delta(D)$. In particular,
$|\delta(D)|=|D|-1$. The coderivation $\delta(D)$ is given by the formula
\begin{multline*}
  \delta(D)[a_1|\ldots|a_k] = \sum_{0\le i<j\le k} (-1)^{(|D|-1)\omega_i} \,
  [a_1|\ldots|D[a_{i+1}|\ldots|a_j]|\ldots|a_k] .
\end{multline*}
Gerstenhaber \cite{Gerstenhaber} introduced the bilinear operation
\begin{equation*}
  (D\circ E) [a_1|\ldots|a_k] = \sum_{0\le i<j\le k} (-1)^{(|E|-1)\omega_i} \,
  D[a_1|\ldots|E[a_{i+1}|\ldots|a_j]| \ldots|a_k]
\end{equation*}
on $G(A)$. The graded commutator
\begin{equation*}
  [D,E] = D \circ E - (-1)^{(|D|-1)(|E|-1)} E\circ D
\end{equation*}
is a graded Lie bracket on $G(A)$, called the Gerstenhaber bracket. We
have
\begin{equation*}
  [\delta(D),\delta(E)] = \delta([D,E]) .
\end{equation*}

\begin{definition}
  An \Ainf-algebra structure on a graded vector space $A$ is a
  codifferential $\delta$ on $BA$.
\end{definition}

Denote the Hochschild cochain associated to the \Ainf-algebra
structure by $m\in G(A)$, so that $\delta=\delta(m)$. Since
$\delta(m)^2=\delta(m\circ m)$, the equation $\delta^2=0$ amounts to the sequence of
quadratic relations $m\circ m=0$ among the homogeneous components
$\{m_{(k)}\}_{k>0}$ of $m$. The differential
\begin{equation*}
  \delta D = [m,D]
\end{equation*}
on $G(A)$ makes the Gerstenhaber algebra of an \Ainf-algebra into a dg
Lie algebra.

The twisting cochain $BA\to A$ given by projection from $BA$ to
$B_1A\cong sA$ followed by the degree one morphism from $sA$ to $A$ is the
universal twisting cochain: the twisting cochains from a dg coalgebra
$C$ to $A$ are in bijection with morphisms of dg coalgebras from $C$
to $BA$.

\begin{definition}
  A morphism of \Ainf-algebras is a morphism of the associated dg
  coalgebras $\f:BA_1\to BA_2$.
\end{definition}

The components $\f_{(k)}:B_kA_1\to A_2$ of a morphism
$\f:BA_1\to BA_2$ of \Ainf-algebras are the compositions of the map
\begin{equation*}
  B_kA_1 \hookrightarrow BA_1 \xrightarrow{\f} BA_2
\end{equation*}
with the universal twisting cochain $BA_2\to A_2$. Together, these
determine $\f$, by the formula
\begin{multline*}
  \f[a_1|\ldots|a_k] = \sum_{\ell=0}^\infty \sum_{0<j_1<\cdots<j_\ell<k} (-1)^{\omega_{j_1}+\cdots+\omega_{j_\ell}} \\
  [\f_{(j_1)}[a_1|\ldots|a_{j_1}]|\f_{(j_2-j_1)}[a_{j_1+1}|\ldots|a_{j_2}]
  \ldots|\f_{(k-j_\ell)}[a_{j_\ell+1}|\ldots|a_k]] .
\end{multline*}

A quasi-isomorphism of \Ainf-algebras $\f:A_1\to A_2$ is a morphism of
\Ainf-algebras such that the linear component $\f_{(1)}:sA_1\to A_2$ (or
rather, the associated morphism from $A_1$ to $A_2$) is a
quasi-isomorphism.

Dg algebras are special cases of \Ainf-algebras, with
$m_{(1)}(a_1)=da_1$ and
\begin{equation*}
  m_{(2)}(a_1,a_2) = (-1)^{|a_1|} a_1a_2 , \quad a_1,a_2\in A .
\end{equation*}
The differential of the bar construction of a dg algebra is the
coderivation
\begin{equation*}
  \delta = \delta_1+\delta_2 = \delta\bigl(m_{(1)}\bigr)+\delta\bigl(m_{(2)}\bigr) ,
\end{equation*}
where
\begin{equation*}
  \delta_1[a_1|\ldots|a_k] = \sum_{j=1}^k (-1)^{\omega_{j-1}} \, [a_1|\ldots|da_j|\ldots|a_k]  
\end{equation*}
is the coderivation associated to the differential $d$ on the complex
$A$, and $\delta_2$ is the coderivation
\begin{equation*}
  \delta_2[a_1|\ldots|a_k] = \sum_{j=1}^{k-1} (-1)^{\omega_j+1} \, [a_1|\ldots|a_ja_{j+1}|\ldots|a_k]
\end{equation*}
associated to the product on $A$.

We learned the following result from Tsygan \cite{Tsygan}. The
associated dg bialgebra structure on $BG(A)$ was discovered by Getzler
and Jones \cite{GJ}. A dg (left) module in coalgebras $M$ for a dg
bialgebra $H$ is a (left) module for $H$ in the monoidal category of
dg coalgebras. In other words, there is an associative action
\begin{equation*}
  H \otimes M \to M
\end{equation*}
which is a morphism of dg coalgebras.
\begin{proposition}
  Let $A$ be an \Ainf-algebra. There is a unique dg bialgebra
  structure on $H=BG(A)$ and $H$-module in dg coalgebras structure on
  $M=BA$ such that the action of $[D]\in BG(A)$ on $BA$ is the graded
  coderivation $\delta(D)$ associated to $D\in G(A)$.
\end{proposition}
\begin{proof}
  The compatibility of the action of $BG(A)$ on $BA$ with the
  coproducts determines the formula for the action of the element
  $[D_1|\ldots|D_n]\in B_nG(A)$ on $[a_1|\ldots|a_k]\in B_kA$:
  \begin{multline*}
    [D_1|\ldots|D_n] \bullet [a_1|\ldots|a_k] = \sum_{0\le i_1 < j_1\le\cdots\le i_n<j_n\le k}
    (-1)^{\sum_{\ell=1}^n (|D_\ell|-1)\omega_{i_\ell}} \\
    [a_1|\ldots|a_{i_1}|D_1[a_{i_1+1}|\ldots|a_{j_1}]|\ldots|a_{i_n}|
    D_n[a_{i_n+1}|\ldots|a_{j_n}]|\ldots|a_k] .
  \end{multline*}
  This gives an injection of graded vector spaces
  \begin{equation*}
    BG(A) \hookrightarrow \End(BA) .
  \end{equation*}
  It may be checked that this subspace is closed under the
  differential $\ad(\delta)$ on $\End(BA)$, giving rise to a codifferential
  on $BG(A)$, and in particular, an \Ainf-algebra structure on
  $G(A)$. This subspace is also closed under composition, giving rise
  to the product on $BG(A)$, and making it into a dg bialgebra.
\end{proof}

The differential on $BG(A)$ induces an \Ainf-algebra structure on
$G(A)$, given by the element
\begin{equation*}
  M[D_1|\ldots|D_n] =
  \begin{cases}
    [m,D_1] , & n=1 , \\
    m\{D_1,\ldots,D_n\} , & n>1 ,
  \end{cases}
\end{equation*}
of $G(G(A))$. In the special case that $A$ is a dg algebra, $G(A)$ is
a dg algebra, with Gerstenhaber's cup product
\begin{equation*}
  D_1 \cup D_2 = (-1)^{|D_1|} \, m\{D_1,D_2\} .
\end{equation*}

The product on $BG(A)$ is determined by its components $m_{k,\ell}$,
which are the compositions of the product
\begin{equation*}
  B_kG(A) \otimes B_\ell G(A) \to BG(A)
\end{equation*}
with the universal twisting cochain $BG(A)\to A$. The map
\begin{equation*}
  m_{k,\ell}: (sG(A))^{\otimes k}\otimes(sG(A))^{\otimes\ell} \to G(A)
\end{equation*}
vanishes unless $k=1$, and the operations $m_{1,\ell}$ are the brace
operations
\begin{equation*}
  m_{1,\ell}\bigl( [D] \otimes [E_1|\ldots|E_\ell] \bigr) = D\{E_1,\ldots,E_\ell\}
\end{equation*}
introduced in \cite{GM}, given by the formula
\begin{multline*}
  D\{E_1,\ldots,E_\ell\} [a_1|\ldots|a_k] = \sum_{0=j_1\le k_1\le\cdots\le j_\ell\le k_\ell\le k}
  (-1)^{\sum_{i=1}^\ell (|E_i|-1)\omega_{j_i}} \\
  D[a_1|\ldots|E_1[a_{j_1+1}|\ldots|a_{k_1}]| \ldots|E_\ell[a_{j_\ell+1}|\ldots|a_{k_\ell}]|\ldots|a_k]
  .
\end{multline*}
The operation $m_{1,1}$ is Gerstenhaber's operation $D\circ E$.

\section{Contractions}

A \textbf{weak contraction} is a pair of complexes $(X,\delta)$ and
$(Y,\partial)$, together with morphisms of complexes $f:X\to Y$ and
$g:Y\to X$ and a homotopy $h:X\to s^{-1}X$ of degree $-1$ such that
$fg=1_Y$ and $gf=1_X-[\delta,h]=1-(\delta h+h\delta)$:
\begin{equation*}
  \begin{tikzcd}
    h \,
    \rotatebox[origin=c]{270}{$\circlearrowright$} \, X
    \arrow[shift left=0.25em]{r}{f} &
    Y \arrow[shift left=0.25em]{l}{g}
  \end{tikzcd}
\end{equation*}
Hence $p=gf$ and $1-p=\delta h+h\delta$ are idempotents.

If the weak contraction $(X,Y,f,g,h)$ satisfies the additional
equations
\begin{align*}
  fh &= 0 , & hg &= 0 , & h^2=0 ,
\end{align*}
we call it a \textbf{contraction}. (These are also referred to as the
side conditions of homological perturbation theory.) If we replace the
homotopy $h$ of a weak contraction by
\begin{equation*}
  \hat{h} = (1-p)h\delta h\delta h(1-p) = \tilde{h} \delta \tilde{h} ,
\end{equation*}
where $\tilde{h}=(1-p)h(1-p)$, we obtain a contraction
$(X,Y,\hat{h},f,g)$.

Let $\Cone(f)=X\oplus sY$ be the mapping cone of $f:X\to Y$, and let
$u$ be a formal commuting variable of degree $2$. We may assemble the
data of a contraction into a curved Maurer--Cartan element $\CA$ on
$\Cone(f)[u]$, with curved differential $\CD+\CA$, where
\begin{align*}
  \CD &= \begin{bmatrix}
    \delta & 0 \\
    0 & -\partial
  \end{bmatrix} \intertext{and}
  \CA &=
  \begin{bmatrix}
    uh & ug \\
    f & 0
  \end{bmatrix} ,
\end{align*}
and curvature $(\CD+\CA)^2=u$.

A \text{perturbation} of a contraction is a Maurer--Cartan element
$\mu\in \Hom(X,X)$ of degree $1$ such that $1+\mu h$ is
invertible. Equivalently, we may assume that $1+h\mu$ is invertible, since
\begin{equation*}
  (1+h\mu)^{-1} = 1 - h(1+\mu h)^{-1}\mu
\end{equation*}
and
\begin{equation*}
  (1+\mu h)^{-1} = 1 - \mu(1+h\mu)^{-1}h ,
\end{equation*}
or that $1+h\mu+\mu h$ is invertible, since
\begin{equation*}
  (1+h\mu+\mu h)^{-1} = (1+h\mu)^{-1}(1+\mu h)^{-1}
\end{equation*}
and
\begin{equation*}
  (1+\mu h)^{-1} = (1+h\mu) (1+h\mu+\mu h)^{-1} .
\end{equation*}

A perturbation gives rise to a new contraction
$(X_*,Y_*,h_*,f_*,g_*)$. Here, $X_*$ has the same underlying graded
vector space as $V$, and differential $\delta_*=\delta + \mu$, and
$W_*$ has the same underlying graded vector space as $W$, with
differential
\begin{equation*}
  \partial_* = \partial + f(1+\mu h)^{-1}\mu g .
\end{equation*}
The remaining data are given by the formulas
\begin{gather*}
  f_* = f(1+\mu h)^{-1} , \quad g_* = (1+h\mu)^{-1}g , \\
  h_* = h(1+\mu h)^{-1} = (1+h\mu)^{-1}h .
\end{gather*}
In terms of the representation of a contraction as a curved
differential $\CD+\CA$ on the mapping cone $C(f)$, the deformed curved
differential equals $\CD + M_*+\CA_*$, where
\begin{equation*}
  M_* = \begin{bmatrix}
    \mu & 0 \\
    0 & 0
  \end{bmatrix} \text{ and }
  \CA_* =
  \begin{bmatrix}
    uh_* & ug_* \\
    f_* & 0
  \end{bmatrix}
  = u\CA(u + M_* \CA)^{-1} .
\end{equation*}
For more on contractions and their perturbations, see \cite{GL}. (What
that paper refers to as strong deformation retract data are what we
call weak contractions.)

Suppose that $A$ is an \Ainf-algebra with differential $\delta$.  Let
$\delta_\mu=\delta+\delta(\mu)$ be the associated codifferential on
$BA$, where $\mu\in G(A)$. Kadeishvili \cite{Kadeishvili} showed that if
$(A,Z,f,g,h)$ is a contraction, there is a natural \Ainf-algebra
structure on $Z$, corresponding to a codifferential $D_*$ on $Z$, and
\Ainf{} quasi-isomorphism $\g_*:BZ\to BA$ whose linearization is
$g$. This \Ainf-algebra structure and \Ainf-morphism are constructed
by solving a fixed-point equation.

Gugenheim, Lambe and Stasheff \cite{GLS} introduced a different
approach to homological perturbation theory for \Ainf-algebras, which
they named the tensor trick: they consider homological perturbation
theory for the bar construction $BA$. In this way, they also obtain a
left inverse $\f_*:BA\to BZ$ to the quasi-isomorphism $\g_*:BZ\to BA$,
with linearization $f$. (See also Huebschmann and Kadeishvili
\cite{HK}.) We now review their results. Note that \cite{Kadeishvili}
and \cite{GLS} restrict their attention to the case $\mu_{(1)}=0$: we
state our results in the more general setting in which
$1+\mu_{(1)}h:A\to A$ is invertible. In other words, we allow the
differential underlying the \Ainf-algebra structure on $A$ to differ
by a Maurer--Cartan element from the original differential on $A$.

The \Ainf-algebra structure on $Z$ and \Ainf-morphism from $Z$ to $A$
constructed by the tensor trick agree with those constructed by
Kadeishvili: this may be proved by showing that they solve the same
fixed-point equations. We will not have need of this identification in
this paper.

The tensor trick is as follows. The contraction $(A,Z,f,g,h)$ induces
a contraction $(BA,BZ,\f,\g,\h)$ of the bar construction $BA$
associated to the underlying complex of $A$ with the \Ainf-algebra
structure with all higher brackets set equal to zero. The homotopy
$\h$ is given by the formula
\begin{equation*}
  \h[a_1|\ldots|a_k] = \sum_{j=1}^k (-1)^{\omega_{j-1}}
  [pa_1|\ldots|pa_{j-1}|ha_j|a_{j+1}|\ldots|a_k] ,
\end{equation*}
and the morphisms $\f:BA\to BZ$ and $\g:BZ\to BA$ are given by the
formulas
\begin{equation*}
  \f[a_1|\ldots|a_k] = [fa_1|\ldots|fa_k]
\end{equation*}
and
\begin{equation*}
  \g[a_1|\ldots|a_k] = [ga_1|\ldots|ga_k] .
\end{equation*}
Thus $\f$ and $\g$ are morphisms of coalgebras:
\begin{equation*}
  \Delta\f = ( \f \otimes \f ) \Delta
\end{equation*}
and
\begin{equation*}
  \Delta\g = ( \g \otimes \g ) \Delta .
\end{equation*}
The idempotent
\begin{equation*}
  \p[a_1|\ldots|a_k] = \g\f[a_1|\ldots|a_k] = [pa_1|\ldots|pa_k]
\end{equation*}
is also a morphism of coalgebras. However, the homotopy operator $\h$
is not a coderivation: rather, it satisfies the formula
\begin{equation*}
  \Delta\h = ( \h \otimes 1 + \p \otimes \h ) \Delta .
\end{equation*}
\begin{lemma}
  If $1+\delta(\mu_{(1)})h:A\to A$ is invertible, then
  $1+\delta(\mu)\h:BA\to BA$ is invertible.
\end{lemma}
\begin{proof}
  On $B_kA$, we have
  \begin{multline*}
    (1+\delta(\mu)\h)^{-1} = \sum_{\ell=0}^\infty \sum_{j_1+\cdots+j_\ell<k} (-1)^\ell
    (1+\delta(\mu_{(1)})\h)^{-1} \delta(\mu_{(j_1+1)})\h \\
    \ldots \bigl( 1+\delta( \mu_{(1)} )\h \bigr)^{-1} \delta(\mu_{(j_\ell+1)})\h
    \bigl( 1+\delta( \mu_{(1)} )\h \bigr)^{-1} .
  \end{multline*}
  Thus, it suffices to prove that
  $1+\delta( \mu_{(1)} )\h : B_kA\to B_kA$ is invertible.

  On $B_kA$, we have
  \begin{equation*}
    \delta( \mu_{(1)} )\h = \sum_{i,j=1}^k \alpha_{ij} ,
  \end{equation*}
  where
  \begin{align*}
    \alpha_{ij} &= \bigl( 1^{\otimes i-1} \otimes \delta( \mu_{(1)} ) \otimes 1^{\otimes k-i} \bigr)
             \bigl( p^{\otimes j-1} \otimes h \otimes 1^{\otimes k-j} \bigr) \\
    &=
    \begin{cases}
      p^{\otimes i-1} \otimes \delta( \mu_{(1)} ) p \otimes p^{\otimes j-i-1} \otimes h \otimes 1^{\otimes k-j} , & i<j , \\
      p^{\otimes j-1} \otimes \delta( \mu_{(1)} ) h \otimes 1^{\otimes k-j} , & i=j , \\
      - p^{\otimes j-1} \otimes h \otimes 1^{\otimes i-j-1} \otimes \delta( \mu_{(1)} ) \otimes 1^{\otimes k-i} , & i>j .
    \end{cases}
  \end{align*}
  Since $\alpha_{pq}\alpha_{ij}=0$ unless $i<j$ or $q\ge i=j$, and
  $\alpha_{pq}\ldots\alpha_{ij}=0$ if $i<j\le q$, we see that
  \begin{align*}
    ( 1 + \delta( \mu_{(1)} ) \h )^{-1}
    &= \Bigl( 1 - \sum_{i>j} \alpha_{ij} \Bigr) \Bigl( 1 + \sum_{i\le j} \alpha_{ij}
      \Bigr)^{-1} \\
    &= \Bigl( 1 - \sum_{i>j} \alpha_{ij} \Bigr) \sum_{\ell=0}^{k-1} (-1)^\ell
    \sum_{\substack{ j_1<\cdots<j_\ell \\ i_1<j_1,\ldots,i_\ell<j_\ell}} \beta \alpha_{i_1j_1} \beta \ldots \beta
    \alpha_{i_\ell j_\ell} \beta
  \end{align*}
  where
  $\beta = \bigl( 1 + \alpha_{kk} \bigr)^{-1} \ldots \bigl( 1 + \alpha_{11} \bigr)^{-1}$.
  But $1+\alpha_{jj}:B_kA\to B_kA$ is invertible:
  \begin{equation*}
    \bigl( 1+\alpha_{jj} \bigr)^{-1} = 1 -
    p^{\otimes j-1} \otimes \delta( \mu_{(1)} )h\bigl( 1 + \delta( \mu_{(1)} )h \bigr)^{-1} \otimes
    1^{\otimes k-j} .
  \end{equation*}
  Hence $\beta$ is invertible, completing the proof.
\end{proof}

By this lemma, if $1+\delta(\mu_{(1)})h:A\to A$ is invertible,
$\delta(\mu)$ induces a deformed contraction
\begin{gather*}
  \f_* = \f(1+\delta(\mu)\h)^{-1} : BA \to BZ , \quad \g_* = (1+\h \delta(\mu))^{-1}\g :
  BZ \to BA , \\
  \h_* = \h(1+\delta(\mu)\h)^{-1} : BA \to BA .
\end{gather*}
The deformed differential $\partial_*$ on $BZ$ is given by the formula
\begin{equation*}
  \partial_* = \f\delta\g + \f \delta(\mu)(1+\h \delta(\mu))^{-1} \g ,
\end{equation*}
and the deformed idempotent $\p_*$ on $BA$ by the formula
\begin{equation*}
  \p_* = (1+\h \delta(\mu))^{-1}\p(1+\delta(\mu)\h)^{-1} .
\end{equation*}

\begin{proposition}
  The maps $\f_*$, $\g_*$ and $\p_*=\g_*\f_*$ are morphisms of
  coalgebras, and $\partial_*$ is a coderivation.
\end{proposition}
\begin{proof}
  Using the formulas $\h\h_*=0$, $\h\p_*=0$, $\p\h_*=0$ and
  $\p\p_*=\p(1+\delta(\mu)\h)^{-1}$, we see that
  \begin{multline*}
    ( \h\otimes1 + \p\otimes\h ) ( \delta(\mu)\otimes1 + 1\otimes\delta(\mu) ) ( \h_* \otimes 1 + \p_* \otimes \h_* ) \\
    \begin{aligned}
      &= \h\delta(\mu)\h_* \otimes 1 + \h\delta(\mu)\p_* \otimes \h_* + \p\delta(\mu)\h_* \otimes \h + \p\p_*
        \otimes \h\delta(\mu)\h_* \\
      &= ( \h - \h_* ) \otimes 1 + ( \p ( 1 + \delta(\mu)\h )^{-1} - \p_* ) \otimes \h_* \\
      & \quad + ( \p - \p ( 1 + \delta(\mu)\h )^{-1} ) \otimes \h + \p ( 1 + \delta(\mu)\h
        )^{-1} \otimes ( \h - \h_* ) \\
      &= ( \h \otimes 1 + \p \otimes \h ) - ( \h_* \otimes 1 + \p_* \otimes \h_* ) .
    \end{aligned}
  \end{multline*}
  It follows that
  \begin{multline*}
    \bigl( 1\otimes1 + ( \h\otimes1 + \p\otimes\h ) ( \delta(\mu)\otimes1 + 1\otimes\delta(\mu) ) \bigr)^{-1} \\
    = 1\otimes1 - ( \h_* \otimes 1 + \p_* \otimes \h_* )( \delta(\mu)\otimes1 + 1\otimes\delta(\mu) ) ,
  \end{multline*}
  proving that
  \begin{align}
    \label{hd}
    \Delta(1+\h\delta(\mu))^{-1}
    &= \bigl( 1\otimes1 + ( \h\otimes1 + \p\otimes\h ) ( \delta(\mu)\otimes1 + 1\otimes\delta(\mu) ) \bigr)^{-1} \Delta \\
    \notag
    &= \bigl( 1\otimes1 - ( \h_* \otimes 1 + \p_* \otimes \h_* )( \delta(\mu)\otimes1 + 1\otimes\delta(\mu) )
      \bigr) \Delta .
  \end{align}
  We see that
  \begin{align*}
    \Delta\g_* &= \Delta (1+\h\delta(\mu))^{-1}\g \\
    &= \bigl( 1\otimes1 - ( \h_* \otimes 1 + \p_* \otimes \h_* )( \delta(\mu)\otimes1 + 1\otimes\delta(\mu) )
      \bigr) ( \g\otimes\g ) \Delta \\
    &= ( \g_*\otimes\g_* ) \Delta .
  \end{align*}

  Since $\f\h=0$, we see that
  \begin{align*}
    \f\g_*
    &= \f(1+\h\delta)^{-1}\g \\
    &= \f \bigl( 1 - \h(1+\delta\h)^{-1}\delta \bigr)\g \\
    &= \f\g = 1 .
  \end{align*}
  It follows that $\partial_*=\f \delta\g+\f\delta(\mu)\g_*$ is a coderivation:
  certainly, $\f\delta\g$ is a coderivation, while
  \begin{align*}
    \Delta \f\delta(\mu)\g_* &= (\f\otimes\f)(\delta(\mu)\otimes1+1\otimes\delta(\mu))(\g_*\otimes\g_*) \Delta \\
              &= ( \f\delta(\mu)\g_*\otimes\f\g_* + \f\g_*\otimes\f\delta(\mu)\g_*) \Delta \\
              &= ( \f\delta(\mu)\g_*\otimes 1 + 1 \otimes\f\delta(\mu)\g_*) \Delta .
  \end{align*}

  The same calculation as for the proof of \eqref{hd} shows that
  \begin{equation*}
    \Delta(1+\delta(\mu)\h)^{-1} = \bigl( 1\otimes1 - ( \delta(\mu)\otimes1 + 1\otimes\delta(\mu) )( \h_* \otimes 1 +
    \p_* \otimes \h_* ) \bigr) \Delta .
  \end{equation*}

  It follows that
  \begin{align*}
    \Delta\f_* &= \Delta \f(1+\delta(\mu)\h)^{-1} \\
    &= ( \f\otimes\f ) \bigl( 1\otimes1 - ( \delta(\mu)\otimes1 + 1\otimes\delta(\mu) ) ( \h_* \otimes 1 + \p_* \otimes
      \h_* ) \bigr) \Delta \\
   &= \bigl( \f_* \otimes \f_* \bigr) \Delta .
   \qedhere
  \end{align*}
\end{proof}

\begin{corollary}
  Let $\delta_\mu=\delta+\delta(\mu)$, $\mu\in G(A)$, be a codifferential on
  $BA$ corresponding to an \Ainf-algebra structure on $A$, such that
  \begin{equation*}
    1+\delta\bigl( \mu_{(1)}\bigr) h : A\to A
  \end{equation*}
  is invertible. The codifferential $\partial_*$ on $BZ$ constructed by the
  tensor trick induces an \Ainf-algebra structure on $Z$, and
  $\f_*:BA\to BZ$ and $\g_*:BZ\to BA$ are quasi-isomorphisms of
  \Ainf-algebras.
\end{corollary}

\section{The bar construction of a symmetric algebra}

The bar construction $BA$ of a dg algebra $A$ is a graded commutative
bialgebra: the product on $BA$ is the shuffle product
\begin{multline*}
  [a_1|\ldots|a_k] \sh [a_{k+1}|\ldots|a_{k+\ell}] \\
  = \sum_{\sigma\in S(k,\ell)} (-1)^{\sum_{i=1}^k ( \omega_{k+\sigma(i)-i} - \omega_k ) (|a_i|-1)}
  [a_{\sigma^{-1}(1)}|\ldots|a_{\sigma^{-1}(k+\ell)}] ,
\end{multline*}
where $S(k,\ell)$ is the set of shuffles,
\begin{equation*}
  S(k,\ell) = \{ \sigma\in S_{k+\ell} \mid \text{$\sigma(i)<\sigma(j)$ if $1\le i<j\le k$ or
    $k+1\le i<j\le k+\ell$} \} .
\end{equation*}
The codifferential $\delta$ of $BA$ is a derivation with respect to the
shuffle product if and only if $A$ is graded commutative.

The calculation of the cohomology of $BSV$ is one of the fundamental
results of homological algebra. The subspace $sV\subset B_1SV$ generates a
dg commutative subalgebra $\Lambda V$ of the dg commutative bialgebra
$BSV$. Since elements of $sV$ are primitive in $BSV$, the inclusion
$g:\Lambda V\hookrightarrow BSV$ is a morphism of dg commutative bialgebras.

In this section, we construct a contracting homotopy
$h:B_kSV\to B_{k+1}SV$ which shows that $g$ is a quasi-isomorphism. Let
$\{x^\alpha\}$ be a homogeneous basis of $V$, let $x^\alpha\in G(SV)$ be the
associated Hochschild zero-cochain, and let
$\overline{\partial}{}_\alpha\in G(SV)$ be the Hochschild one-cochain
\begin{equation*}
  (-1)^{|x^\alpha|} \overline{\partial}{}_\alpha a = \partial_\alpha a - \epsilon(\partial_\alpha a) ,
\end{equation*}
where $\partial_\alpha a$ is the graded partial derivative of $a$ with respect to
$x^\alpha$, and $\epsilon(\partial_\alpha a)$ is the constant term of $\partial_\alpha a$.

Let $\rho\in G(SV)$ be the Hochschild one-cochain which acts by
multiplication by $k$ on $S^kV$. Let
\begin{equation*}
  \tau[a] = \sum_\alpha \epsilon(\partial_\alpha a) x^\alpha  \in G(A)
\end{equation*}
be the Hochschild one-cochain that projects from $SV$ to
$V=S^1V\subset SV$. We also denote the coderivations $\delta(\rho)$ and
$\delta(\tau):BSV\to BSV$ by $\rho$ and $\tau$. Since the product in
$SV$ is homogeneous of degree $0$, it is clear that $\rho$ commutes with
the differential $\delta$. We have
\begin{equation}
  \label{rhotau}
  \sum_\alpha x^\alpha \cup \overline{\partial}{}_\alpha = \rho - \tau \in G(SV) .
\end{equation}

The intersection of $\Wedge V$ with $B_kSV$ is the image of the
projection
\begin{equation}
  \label{pi}
  p_k[a_1|\ldots|a_k] = \frac{1}{k!} \, [\tau a_1] \sh \cdots \sh [\tau a_k]
\end{equation}
on $B_kSV$. Note that $p$ is a morphism of dg algebras, though not of
coalgebras. Denote the map $g^{-1}\circ p$ from $BSV$ to $\Wedge V$ by
$f$: thus $fg=1_{\Wedge V}$, and $gf=p$.

Define the operators $\xi:B_kSV\to B_{k+1}SV$ and
$\lambda:B_kSV\to B_kSV$ by the formulas
\begin{align*}
  \xi[a_1|\ldots|a_k]
  &= \sum_\alpha \bigl[x^\alpha\bigm|\overline{\partial}{}_\alpha\bigr] \bullet [a_1|\ldots|a_k] \\
  &= \sum_{0\le i<j\le k} \sum_\alpha (-1)^{\omega_i+|x^\alpha|(\omega_{j-1}-\omega_i+1)} \\
  & \qquad\qquad [a_1|\ldots|a_i|x^\alpha|\ldots|\partial_\alpha a_j - \epsilon(\partial_\alpha a_j)|\ldots|a_k] , \\
  \lambda[a_1|\ldots|a_k] &= [a_1|\ldots|a_{k-1}] \sh [\tau a_k] .
\end{align*}
By inspection, it is clear that the operators $\xi$ and $\lambda$ both commute
with $\rho$.

We learned the following result from \cite{dWL}.
\begin{proposition}
  \label{dWL}
  $[\delta,\xi] = \rho - \lambda$
\end{proposition}
\begin{proof}
  Introduce the matrix for the differential $d:V\to V$ of $V$:
  \begin{equation*}
    dx^\alpha = \sum_\beta M^\alpha_\beta x^\beta .
  \end{equation*}
  The differential $[\delta,\xi]$ of $\xi$ is the sum of $[\delta_1,\xi]$ and
  $[\delta_2,\xi]$. The first of these vanishes:
  \begin{equation*}
    \sum_\alpha \delta_1 \bigl[ x^\alpha| \overline{\partial}{}_\alpha \bigr] =
    \sum_{\alpha,\beta} \Bigl( \bigl[ M^\alpha_\beta x^\beta \bigm| \overline{\partial}{}_\alpha \bigr] -
    \bigl[ x^\alpha \bigm| M^\beta_\alpha \overline{\partial}{}_\beta \bigr] \Bigr) = 0 .
  \end{equation*}
  The second equals the action of the element
  \begin{equation*}
    \sum_\alpha \delta_2 \bigl[ x^\alpha| \overline{\partial}{}_\alpha \bigr]
    = \sum_\alpha \biggl( \bigl[x^\alpha \cup \overline{\partial}{}_\alpha\bigr]
    + \bigl[ m_{(2)} \circ x^\alpha \bigm| \overline{\partial}{}_\alpha\bigr]
    - (-1)^{|x^\alpha|}  \bigl[ x^\alpha \bigm| [ m_{(2)} , \overline{\partial}{}_\alpha ]
    \bigr] \biggr)
  \end{equation*}
  of $BG(SV)$ on $BSV$. The first term equals $\rho-\tau$ by \eqref{rhotau},
  while the second term vanishes since $SV$ is graded commutative.
  The two-cochain $[m_{(2)} , \overline{\partial}{}_\alpha ] \in G(SV)$ equals
  \begin{equation*}
    \bigl( [ m_{(2)} , \overline{\partial}{}_\alpha ] \bigr)[a_1|a_2] =
    - (-1)^{|x^\alpha|} \, \epsilon(\partial_\alpha a_1) a_2 + (-1)^{|a_1||x^\alpha|} a_1 \epsilon(\partial_\alpha
    a_2) .
  \end{equation*}
  It follows that
  \begin{multline*}
    \sum_\alpha \bigl[ x^\alpha \bigm| [ m_{(2)} , \overline{\partial}{}_\alpha ] \bigr] \bullet
    [a_1|\ldots|a_k] \\
    \begin{aligned}
      &= \sum_{1\le i\le j<k} (-1)^{(|a_j|+1)(\omega_{j-1}-\omega_{i-1})}
        [a_1|\ldots|a_{i-1}|\tau a_j|a_i|\ldots|\widehat{a}{}_j|\ldots|a_k] \\
      &\quad - \sum_{1\le i\le j<k} (-1)^{(|a_{j+1}|+1)(\omega_j-\omega_{i-1})}
        [a_1|\ldots|a_{i-1}|\tau a_{j+1}|a_i|\ldots|\widehat{a}{}_{j+1}|\ldots|a_k] \\
      &= ( \tau - \lambda ) [a_1|\ldots|a_k] .
    \end{aligned}
  \end{multline*}
  The result follows.
\end{proof}

\begin{corollary}
  $[\delta,\lambda]=0$
\end{corollary}
\begin{proof}
  We have
  \begin{equation*}
    [\delta,\lambda] = [\delta,\rho - (\rho-\lambda) ] = [\delta,\rho] - [\delta,[\delta,\xi]] .
  \end{equation*}
  The second term vanishes since $\delta^2=0$.
\end{proof}

\begin{proposition}
  \label{derivation}
  \mbox{}
  \begin{enumerate}[1)]
  \item $\xi^2=0$
  \item The operator $\xi$ is a derivation of the shuffle product on
    $BSV$.
  \end{enumerate}
\end{proposition}
\begin{proof}
  We have
  \begin{equation}
    \label{Delta xi}
    \Delta \xi = \sum_\alpha [x^\alpha] \otimes [\overline{\partial}{}_\alpha] .
  \end{equation}
  Squaring both sides of \eqref{Delta xi}, we see that
  $\xi^2\in BG(SV)$ is primitive: $\Delta\xi^2=0$. We prove that
  $\xi^2f=0$ for $f\in B_kSV$ by induction on $k$: it is clearly true for
  $k=1$. By the induction hypothesis, $\xi^2f\in B_{k+2}SV$ is
  primitive. But the space of primitive elements of $SV$ equals
  $B_1SV$, hence $\xi^2f=0$.

  Let $\zeta(f,g)=\xi(f\sh g)-(\xi f)\sh g-(-1)^{|f|}\,f\sh(\xi g)$. By
  \eqref{Delta xi}, we see that for all $f,g\in BSV$,
  \begin{multline*}
    \Delta \zeta(f,g) = (-1)^{|f^{(2)}||g^{(1)}|} \Bigl( \zeta\bigl(
    f^{(1)},g^{(1)} \bigr) \otimes \bigl( f^{(2)}\sh g^{(2)} \bigr) \\
    + (-1)^{|f^{(1)}|+|g^{(1)}|} \bigl( f^{(1)}\sh g^{(1)} \bigr) \otimes
    \zeta\bigl( f^{(2)},g^{(2)} \bigr) \Bigr) .
  \end{multline*}
  Here, the possible cross-terms cancel, because
  $[x^\alpha]\bullet (f\sh g) = [x^\alpha] \sh (f\sh g)$ and
  \begin{equation*}
    [\overline{\partial}{}_\alpha]\bullet (f\sh g) = ( [\overline{\partial}{}_\alpha]\bullet f ) \sh g +
    (-1)^{|x^\alpha||f|} f \sh ( [\overline{\partial}{}_\alpha]\bullet g) .
  \end{equation*}
  We prove that $\zeta(f,g)=0$ for $f\in B_kSV$ and $g\in B_\ell SV$ by induction
  on $k+\ell$; it is clearly true for $k+\ell=2$. By the induction
  hypothesis, $\zeta(f,g) \in B_{k+\ell+1}SV$ is primitive, hence it vanishes.
\end{proof}

\begin{corollary}
  $[\xi,\lambda]=0$
\end{corollary}
\begin{proof}
  We have
  \begin{equation*}
    [\xi,\lambda] = [\xi,\rho - (\rho-\lambda) ] = [\xi,\rho] - [\xi,[\delta,\xi]] .
  \end{equation*}
  The first term has been seen to vanish, while the second term
  equals $\frac12[\delta,\xi^2]=0$.
\end{proof}

\begin{corollary}
  The operator $\lambda$ is a derivation with respect to the shuffle
  product:
  \begin{equation*}
    \lambda(x\sh y) = \lambda x\sh y + x\sh\lambda y .
  \end{equation*}
\end{corollary}
\begin{proof}
  Since $SV$ is a graded commutative algebra, $\delta$ is a graded
  derivation with respect to the shuffle product, hence so is the
  graded commutator $[\delta,\xi]$. But it is clear that $\rho$ is a derivation
  with respect to the shuffle product, and the result follows.
\end{proof}

\begin{lemma}
  The descending factorial
  \begin{equation*}
    (\lambda)_j = \lambda(\lambda-1)\ldots(\lambda-j+1)
  \end{equation*}
  is given by the formula
  \begin{equation*}
    (\lambda)_j[a_1|\ldots|a_k] = [a_1|\ldots|a_{k-j}] \sh [\tau a_{k-j+1}] \sh \cdots \sh [\tau
    a_k] .
  \end{equation*}
\end{lemma}
\begin{proof}
  This is proved by induction on $j$, using the explicit formula for
  $\lambda$, the formula $\lambda[x^\alpha]=[x^\alpha]$, and the fact that $\lambda$ is a
  derivation for the shuffle product.
\end{proof}

\begin{corollary}
  \label{minimal}
  Let $\lambda_k$ and $p_k$ be the restrictions of $\lambda$ and
  $p$ to $B_kSV$. Then
  \begin{equation*}
    (\lambda_k)_k = k!\,p_k ,
  \end{equation*}
  while $(\lambda_k)_i=0$ for $i>k$.
\end{corollary}

In particular, we recover one of the results of de~Wilde and Lecomte,
that the minimal polynomial of $\lambda_k$ divides
$\lambda_k(\lambda_k-1)\ldots(\lambda_k-k)$, and $\lambda_k$ is a semisimple endomorphism whose
spectrum is contained in (and in fact equals) $\{0,\dots,k\}$. The
eigenspace $\{\lambda_k=k\}$ is $\Lambda_kV\subset B_kSV$. The spectrum of the
restriction $\rho_k$ of $\rho$ to $B_kSV$ equals $\{k,k+1,\ldots\}$, and the
subspace $\{\lambda_k=k\}$ is a subset of $\{\rho_k=k\}$. Since $\rho$ and
$\lambda$ commute, the eigenvalues of $\rho_k-\lambda_k$ are nonnegative integers,
and the kernel of $\rho_k-\lambda_k$ is the subspace on which
$\rho_k=\lambda_k=k$, namely $\Lambda^kV$.

We now modify the operator $\xi$ to obtain a homotopy
\begin{equation*}
  h = (\rho-\lambda)^{-1}\xi .
\end{equation*}
This operator is defined because $p\xi=0$, and the eigenvalues of
$\rho-\lambda$ on the image of the idempotent $1-p$ are strictly positive. It
is clear that $h^2=0$, since $\xi^2=0$. It follows that $hp=ph=0$.

We have proved the following theorem.
\begin{theorem}
  The morphism of dg commutative bialgebras
  $g:\Lambda V\hookrightarrow BSV$, the morphism of dg commutative algebras
  $f:BSV \to \Wedge V$, and the homotopy $h=(\rho-\lambda)^{-1}\xi$, form a
  contraction $(BSV,\Wedge V,f,g,h)$
\end{theorem}

In the remainder of this section, we give an explicit formula for the
restriction $h_k:B_kSV\to B_{k+1}SV$ of the homotopy $h$ to $B_kSV$; it
is a finite sum, though the number of terms increases with $k$.
\begin{proposition}
  \begin{equation}
  \label{explicit}
    h_k [a_1|\ldots|a_k] = \sum_{j=1}^k \frac{ \xi[a_1|\ldots|a_j] } {\rho(\rho+1)\ldots(\rho+k-j)}
    \sh [\tau a_{j+1}] \sh \cdots \sh [\tau a_k]
  \end{equation}
\end{proposition}
\begin{proof}
  Since the eigenvalues of the action of $\rho$ on $B_{k+1}SV$ lie in
  $\{k+1,k+2,\ldots\}$, the polynomial $(\rho)_{j+1}$ is invertible on
  $B_{k+1}SV$ for $0\le j\le k$. We have
  \begin{align*}
    (\rho-\lambda) (\rho)_{j+1}^{-1}(\lambda)_j \xi &= \bigl( (\rho-j) - (\lambda-j) \bigr)
                                  (\rho)_{j+1}^{-1} (\lambda)_j \xi \\
    &= (\rho)_j^{-1} (\lambda)_j \xi - (\rho)_{j+1}^{-1} (\lambda)_{j+1} \xi .
  \end{align*}
  It follows that
  \begin{align*}
    (\rho-\lambda) \sum_{j=0}^k (\rho)_{j+1}^{-1}(\lambda)_j \xi
    &= \sum_{j=0}^k \biggl( (\rho)_j^{-1} (\lambda)_j \xi - (\rho)_{j+1}^{-1} (\lambda)_{j+1}
      \xi \biggr) \\
    &= \xi - (\rho)_{k+1}^{-1} (\lambda)_{k+1} \xi .
  \end{align*}
  Since $(\lambda)_{k+1}$ vanishes on $B_kSV$, we see that
  $(\lambda)_{k+1}\xi=\xi(\lambda)_{k+1}=0$, proving the formula
  \begin{equation*}
    h_k = \sum_{j=0}^k (\rho)_{j+1}^{-1}(\lambda)_j \xi .
  \end{equation*}
  Since $\xi$ is a graded derivation for the shuffle product, we see
  that
  \begin{equation*}
    \xi(\lambda)_j [a_1|\ldots|a_k] = \xi[a_1|\ldots|a_{k-j}] \sh [\tau a_{k-j+1}] \sh \cdots \sh
    [\tau a_k] .
  \end{equation*}
  Since
  \begin{multline*}
    \rho \Bigl( [a_1|\ldots|a_{k-j}] \sh [\tau a_{k-j+1}] \sh \cdots \sh
    [\tau a_k] \Bigr) \\
    = \Bigl( (\rho+j) [a_1|\ldots|a_{k-j}] \Bigr) \sh [\tau a_{k-j+1}] \sh \cdots \sh
    [\tau a_k] ,
  \end{multline*}
  the proposition follows.
\end{proof}

Halbout \cite{Halbout} and Dippel et al.\ \cite{DESW} have studied
contracting homotopies on $BSV$ that extend to other classes of
function algebras, such as smooth and analytic functions, when
$\F=\R$. We see from the above formula for $h$ that it may also be
applied in these more general cases. We discuss only the smooth case,
but other function algebras are handled in the same way.

Let $U\subset V$ be a star-shaped open subset of the finite-dimensional real
vector space $V$. Substitite for the algebra $SV$ the algebra
$C^\infty_*(U)$ of smooth functions on $U$ vanishing at $0\in U$, and the bar
construction by the direct sum
\begin{equation*}
  B_\infty C^\infty_*(U) = \bigoplus_{k=1}^\infty s^k C_*^\infty(U,k) ,
\end{equation*}
where $C^\infty_*(U,k)$ is the subspace of smooth functions on $U^k$ that
vanish at $(u_1,\ldots,u_k)\in U^k$ if any of the $u_i$ equal
$0\in U$. The homotopy $h$ defines a continuous linear morphism from
$C^\infty_*(U,k)$ to $C^\infty_*(U,k+1)$; the definitions of the analogues of
$\tau$ and $\overline{\partial}{}_\alpha$ are clear, and the operation
$(\rho+\ell)^{-1}$ is given by the integral
\begin{equation*}
  (\rho+\ell)^{-1} f(u_1,\ldots,u_k) = \int_0^1 f(tu_1,\ldots,tu_k) \, t^{\ell-1} \, dt .
\end{equation*}
This clearly defines a continuous linear map from $C^\infty_*(U,k)$ to
itself if $\ell>0$, while for $\ell=0$, we have
\begin{align*}
  \rho^{-1} f(u_1,\ldots,u_k) &= \sum_{i=1}^k u_i \cdot (\rho+1)^{-1} \int_0^1
                        (\partial_{u_i}f)(su_1,\ldots,su_k) \, ds \\
  &= \sum_{i=1}^k u_i \cdot \int_0^1 \int_0^1
  (\partial_{u_i}f)(stu_1,\ldots,stu_k) \, ds \, dt \in C^\infty_*(U,k) .
\end{align*}
Since $B_kSV$ is dense in $s^kC^\infty_*(U,k)$, it is clear that all of the
formulas that we have proved in the algebraic setting hold for
function algebras.

\section{The cobar construction of an exterior coalgebra}

In this section, we translate our results on the bar construction of a
symmetric algebra to the dual situation, the cobar construction of an
exterior coalgebra: we construct a contraction from the cobar
construction $\Omega\Wedge V$ of an exterior coalgebra to the symmetric
algebra $SV$. Using homological perturbation theory, this contraction
is deformed to a contraction from the cobar construction $\Omega CL$ of the
Chevalley--Eilenberg complex of a dg Lie algebra $L$ to the universal
enveloping algebra $UL$ of $L$: the proof follows Baranovsky
\cite{Baranovsky}, with the difference that we work with an explicit,
and easily computable, formula for the homotopy contraction from
$\Omega\Wedge V$ to $SV$.

Let $C$ be a dg coalgebra. The cobar construction of $C$ is a
dg algebra whose underlying graded algebra is the tensor algebra
\begin{equation*}
  \Omega C = \bigoplus_{k=1}^\infty (s^{-1}C)^{\otimes k} .
\end{equation*}
We denote the element $s^{-1}a_1\otimes \cdots \otimes s^{-1}a_k$ by
$\<a_1|\cdots|a_k\>$. The differential of $\Omega C$ is
\begin{multline*}
  \delta\<a_1|\cdots|a_k\> \\
  = \sum_{j=1}^k (-1)^{\omega_{j-1}} \<a_1|\cdots|da_j|\cdots|a_k\> + \sum_{j=1}^k
  (-1)^{\omega_j+1} \<a_1|\cdots|a_j^{(1)}|a_j^{(2)}|\cdots|a_k\> .
\end{multline*}

The cobar algebra $\Omega C$ is a graded cocommutative bialgebra, with
coproduct given by a sum over coshuffles:
\begin{multline*}
  \<a_1|\cdots|a_k\> \mapsto \sum_{0<\ell<k} \sum_{\sigma\in S(\ell,k-\ell)} (-1)^{\sum_{i=1}^\ell (
    \omega_{\ell+\sigma(i)-i} - \omega_\ell ) (|a_i|-1)} \\
  \<a_{\sigma(1)}|\ldots|a_{\sigma(\ell)}\> \otimes \<a_{\sigma(\ell+1)}|\ldots|a_{\sigma(k)}\> .
\end{multline*}
The differential $\delta$ on $\Omega C$ is a coderivation for this coproduct if
and only if the graded coalgebra $C$ is cocommutative. The subcomplex
$P\Omega C$ of primitive elements of $\Omega C$ is the Harrison chain complex of
$C$: it is the Koszul dual dg Lie algebra to $C$.

As we have seen in the last section, the natural inclusion
$g:\Wedge V \hookrightarrow BSV$ of the exterior coalgebra of a cochain complex
$V$ into the bar construction $BSV$ of its symmetric algebra is a
morphism of dg commutative bialgebras. Similarly, there is a
surjective morphism of dg cocommutative bialgebras
$f:\Omega\Wedge V\to SV$, induced by projecting the generators
$s\Wedge V$ of $\Omega\Wedge V$ to the summand $s\Wedge_1V\cong V$, followed by
inclusion as the generators $V$ of $SV$. This morphism has a section
$g:SV\to\Omega\Wedge V$, which is the morphism of dg cocommutative coalgebras
induced by the inclusions
\begin{equation*}
  S^kV \hookrightarrow V^{\otimes k}\cong\Omega^k\Wedge_1V \hookrightarrow \Omega^k\Wedge V .
\end{equation*}

Let $\rho$ be the coderivation of $\Lambda V$ which acts by multiplication by
$k$ on the subspace $\Lambda_kV\subset \Lambda V$. This coderivation extends to a graded
derivation of $\Omega\Wedge V$:
\begin{equation*}
  \rho\<a_1|\cdots|a_k\> = \sum_{j=1}^k \<a_1|\cdots|\rho a_j|\cdots|a_k\> .
\end{equation*}
We see that $[\delta,\rho]=0$ on $\Omega\Wedge V$.

Let $\tau$ be the projection $\tau$ from $\Wedge V$ to
$\Wedge_1V\cong s^{-1}V$. Define maps
$\xi:\Omega^k\Wedge V\to\Omega^{k-1}\Wedge V$ and
$\lambda:\Omega^k\Lambda V\to\Omega^k\Wedge V$ by the formulas
\begin{align*}
  \xi\<a_1|\cdots|a_k\> &= \sum_{1\le i<j\le k} (-1)^{\omega_{i-1}+|a_i|(\omega_{j-1}-\omega_i+1)}
  \<a_1|\cdots|\widehat{a}_i|\cdots|\tau a_i\wedge a_j|\ldots|a_k\> , \\
  \lambda\<a_1|\cdots|a_k\> &= \sum_{j=1}^k (-1)^{(|a_j|-1)(\omega_k-\omega_{j+1})} \,
  \<|a_1|\cdots|\widehat{a}_j|\cdots|a_k|\tau a_j\> .
\end{align*}
The operators $\xi$ and $\lambda$ both commute with $\rho$.

The proofs of the last section apply to any additive symmetric
monoidal category tensored over the symmetric monoidal category of
finite-dimensional cochain complexes. (The proofs take place in finite
truncations $\bigoplus_{1\le k\le N} B_kA$ and
$\bigoplus_{1\le k\le N} S_kV$ of the bar construction and symmetric
algebra, so the proofs deal only in finite sums.) The opposite of the
category of cochain complexes is such a category, and this
substitution has the effect of exchanging the symmetric algebra with
the exterior coalgebra (after suspension of $V$), and the bar
construction with the cobar construction. In this way, we obtain the
following dual results.
\begin{enumerate}[a)]
\item $[\delta,\xi] = \rho - \lambda$
\item $\xi$ and $\lambda$ are (graded) coderivations with respect to the
  shuffle coproduct on $\Omega\Wedge V$.
\item $\xi^2=0$
\item $[\delta,\lambda]=[\xi,\lambda]=0$
\item Let $\rho_k$ and $\lambda_k$ be the restrictions of $\lambda$ and
  $\rho$ to $\Omega^k\Lambda V$. The operators $\rho_k$ and
  $\lambda_k$ are commuting semisimple operators.
\item The eigenvalues of $\rho_k$ lie in $\{k,k+1,\ldots\}$.
\item The eigenvalues of $\lambda_k$ lie in $\{0,\ldots,k\}$,
  $(\lambda_k)_k=k!\,p_k$, and $(\lambda_k)_i=0$ for $i>k$.
\end{enumerate}

We have proved the following theorem.
\begin{theorem}
  The morphism of dg cocommutative coalgebras
  $g:SV\hookrightarrow \Omega\Wedge V$, the morphism of dg commutative bialgebras
  $f:\Omega\Wedge V\to SV$, and the homotopy
  \begin{equation*}
    h = (\rho-\lambda)^{-1}\xi ,
  \end{equation*}
  form a contraction $(\Omega\Wedge V,SV,f,g,h)$.
\end{theorem}

The restriction $h_k:\Omega^{k+1}\Wedge V\to\Omega^k\Wedge V$ of the homotopy
$h$ to $\Omega^{k+1}\Lambda V$ (that is, corestriction to
$\Omega^k\Wedge V$) is again given by a finite sum
\begin{equation*}
  h_k = \sum_{j=0}^k (\rho)_{j+1}^{-1} (\lambda)_j\xi .
\end{equation*}

Now suppose that $L$ is an \Linf-algebra. The differential $\delta$ of
$CL$ is the sum of the differential $d$ induced by the differential on
the underlying cochain complex $L$ and the contribution of the
brackets $[v_1,\ldots,v_k]$, $k\ge2$. Denote the codifferential
$\delta-d$ of $CL$ by $\mu$: it is a Maurer--Cartan element in the dg Lie
algebra of coderivations of $\Wedge L$, and induces a Maurer--Cartan
element on $\Omega\Wedge L$, deforming it to the dg algebra
$\Omega CL$. Applying the formulas of homological perturbation theory to
the contraction $h$ from $\Omega\Wedge L$ to the subcomplex $SL$, we obtain
a contraction
\begin{equation*}
  h_* = h(1+\mu h)^{-1} : \Omega CL \to s^{-1}\Omega CL
\end{equation*}
from $\Omega CL$ to $SL$, with
\begin{equation*}
  f_* = f(1+\mu h)^{-1} : \Omega CL \to SL .
\end{equation*}
Since $\mu g=0$, the morphism $g=g_*$ of $SL$ into $\Omega CL$ and the
differential $d_*=d$ of $SL$ are not deformed.

By the tensor trick, we obtain a contraction $\h_*$ from $B\Omega CL$ to
$BSL$, which induces a codifferential on $BSL$, that is, an
\Ainf-algebra structure on $SL$. This \Ainf-algebra structure was
introduced by Baranovsky \cite{Baranovsky} (though without an explicit
choice of contraction $h$ from $\Omega\Wedge L$ to $SL$); he calls it the
universal enveloping \Ainf-algebra of the \Linf-algebra $L$.  In this
way, we obtain an analogue of the Poincar\'e--Birkhoff--Witt Theorem for
\Linf-algebras. In the case where $L$ is a dg Lie algebra, we now
identify this \Ainf-algebra structure on $SL$ with the usual
enveloping algebra, following Baranovsky.

The differential on $B\Omega CL$ is a sum
\begin{equation*}
  \delta_B + \delta_\Omega + \delta_1 + \mu ,
\end{equation*}
where $\delta_B$ is the codifferential on $B\Omega CL$ induced by the product of
$\Omega CL$, $\delta_\Omega$ is the differential on $\Omega CL$ induced by the coproduct
of $CL$, $\delta_1$ is induced by the differential on $L$, and $\mu$ is the
coderivation on $CL$ corresponding to the brackets $[x_1,\ldots,x_k]$,
$k\ge2$, on $L$. In applying the tensor trick, the contraction $\h$ is
associated to the complex $B\Omega CL$ with differential
$\delta_\Omega+\delta_1$, and it is perturbed by $\delta_B+\mu$.

Since $\delta_\Omega\g=\mu g=0$, the codifferential on $BSL$ induced by the tensor
trick is
\begin{multline}
  \label{UL}
  \f\bigl( \delta_\Omega+\delta_1 \bigr)\g + \f(1+(\delta_B+\mu)\h)^{-1}(\delta_B+\mu)\g \\
  = \delta_1 + \sum_{k=2}^\infty \f \bigl( - (1+\mu\h)^{-1}\delta_B \h \bigl)^{k-2}
  (1+\mu\h)^{-1}\delta_B\g .
\end{multline}
The $k$-linear bracket $m_k$ of the \Ainf-algebra structure induced on $SL$ by
this codifferential is contributed by the summand indexed by $k$.
\begin{lemma}
  \label{F}
  If $L$ is a dg Lie algebra, that is, $\delta_k=0$ for $k>2$, then this
  sum simplifies to
  \begin{equation*}
    \delta_1 + \f (1+\mu\h)^{-1}\delta_B \g .
  \end{equation*}
\end{lemma}
\begin{proof}
  There is a decreasing filtration on $CL$ by subspaces
  \begin{equation*}
    F^kCL = \bigoplus_{\ell\ge k+1} C_\ell L .
  \end{equation*}
  This induces decreasing filtrations on $\Omega CL$ and $B\Omega CL$. The
  operator $\h$ has degree $1$ for this filtration, the operator
  $\delta_B$ has degree $0$, the operator $\mu$ has degree $-1$ and the
  morphism $\f$ vanishes on $F^1B\Omega CL$. It follows that the operator
  $(1+\mu\h)^{-1}\delta_B \h$ raises filtration degree, and the result
  follows.
\end{proof}

It follows that the \Ainf-algebra structure on $SL$ induced by the
contraction on $B\Omega CL$ is a dg algebra structure: we denote the
resulting product by $x\ast y$, and this deformation of $SL$ by
$S_\ast L$. To identify the product on $S_\ast L$, we consider the
decreasing filtration on $\Omega CL$ induced by the decreasing filtration
\begin{equation*}
  G_kCL = \bigoplus_{\ell\le k} C_\ell L .
\end{equation*}
This filtration is preserved by $\h$ and $\delta_B$, and lowered by
$\mu$. The morphisms $f$ and $g$ are compatible with this filtration and
the decreasing filtration
\begin{equation*}
    G_kSL = \bigoplus_{\ell\le k} S^\ell L .
\end{equation*}
The contribution of the summand $\f\delta_B\g$ to the deformed
codifferential on $BSL$ corresponds to the original product on $SL$,
while the contribution of the remainder of the codifferential
\begin{equation*}
  \f(1+\mu\h)^{-1}\delta_B\g - \f\delta_B\g = - \f(1+\mu\h)^{-1}\mu\h\delta_B\g
\end{equation*}
maps $G_kSL$ to $G_{k-1}SL$. Thus, the induced product on $SL$ may be
characterized by its value on $[x|y]\in sL\otimes sL\subset B_2SL$. This is
calculated as follows:
\begin{multline*}
  \f(1+\mu\h)^{-1}\mu\h\delta_B\g[x|y] = \tfrac{1}{2} \f(1+\mu\h)^{-1}\mu\h\delta_B
  \Bigl( [\<x\>|\<y\>] + (-1)^{|x||y|} [\<y\>|\<x\>] \Bigr) \\
  \begin{aligned}
    &= \tfrac{1}{2} \f(1+\mu\h)^{-1}\mu\h \Bigl( (-1)^{|x|+1} [\<x|y\>] -
    (-1)^{|x||y|+|y|+1} [\<y|x\>] \Bigr) \\
    &= \tfrac{1}{2} \f(1+\mu\h)^{-1}\mu \Bigl( (-1)^{|x|+1} [\<x\wedge y\>] -
    (-1)^{|x||y|+|y|+1} [\<y\wedge x\>] \Bigr) \\
    &= \tfrac{1}{2} \f(1+\mu\h)^{-1} \Bigl( - [\<[x,y]\>] +
    (-1)^{|x||y|} [\<[y,x]\>] \Bigr) \\
    &= - \f [\<[x,y]\>] = - [x,y] .
  \end{aligned}
\end{multline*}
We have proved the following variant of the Poincar\'e--Birkhoff--Witt
theorem. (We have used a different normalization of the codifferential
on $CL$ to the one in Baranovsky's paper.)
\begin{theorem}
  \label{Baranovsky}
  Let $L$ be a dg Lie algebra. There is an isomorphism of dg algebras
  from the universal enveloping algebra $UL$ to $S_\ast L$, defined on
  $L\subset UL$ by $x\mapsto\frac12x$.
\end{theorem}

\section{Application to the non-commutative Gauss--Manin
  connection}

Let $E=Ss^{-1}\F\cong\F\oplus\epsilon\F$, where $\epsilon$ is a variable of degree
$1$ and square zero. Let $\partial_\epsilon=\partial/\partial\epsilon$ be the graded derivation of
$\Lambda$, of degree $-1$:
\begin{equation*}
  \partial_\epsilon(a+\epsilon b) = b .
\end{equation*}
Denote by $L_\epsilon$ the dg Lie algebra $L\otimes E$. Extend
$\partial_\epsilon$ to a coderivation of the dg coalgebra $CL_\epsilon$.

Let $\F[u]$ be the polynomial algebra generated by an element $u$ of
degree $2$. Let
$\Omega_u CL_\epsilon=\bigl(\Omega CL_\epsilon\bigr)[u]$. Deform the differential of the dg
bialgebra $\Omega_uCL_\epsilon$ by the locally finite sum
\begin{equation*}
  \nu = \nu_1 + \nu_2 + \ldots ,
\end{equation*}
where
\begin{equation*}
  \nu_\ell\<a_1|\ldots|a_k\> = \sum_{j=1}^k (-1)^{\omega_{j-1}} u \<a_1|\ldots|
  \partial_\epsilon a_j^{(1)}|\ldots|\partial_\epsilon a_j^{(\ell)}|\ldots|a_k\> .
\end{equation*}
Write $\nu=\nu_1+\nu_+$. The differential of $\Omega_uCL_\epsilon$ is given by the
formula
\begin{multline*}
  \delta\<a_1|\cdots|a_k\> \\
  = \sum_{j=1}^k (-1)^{\omega_{j-1}} \<a_1|\cdots|da_j|\cdots|a_k\>
  + \sum_{j=1}^k (-1)^{\omega_j+1} \<a_1|\cdots|a_j^{(1)}|a_j^{(2)}|\cdots|a_k\> \\
  + u \sum_{\ell=1}^\infty \sum_{j=1}^k (-1)^{\omega_{j-1}}
  \<a_1|\ldots| \partial_\epsilon a_j^{(1)}|\ldots|\partial_\epsilon a_j^{(\ell)}|\ldots|a_k\> .
\end{multline*}
This is the dg algebra denoted
$B^{\text{tw}}\mathfrak{g}^\bullet[u,\epsilon]$ in \cite{Tsygan}.

In this section, using the methods of \S5, we construct an explicit
\Ainf{} quasi-iso\-morph\-ism from $UL$ to $\Omega_uCL_\epsilon$, reproving
Lemma~19 of \cite{Tsygan}: this proof yields an explicit formula for
the twisting cochain realizing the Gauss-Manin connection on periodic
cyclic homology whose existence was proved loc.\ cit.

Define an operator $\epsilon:CL_\epsilon\to CL_\epsilon$ of degree $1$ by the formula
\begin{equation*}
  \epsilon \bigl( x_1 \wedge \ldots \wedge x_n \bigr) = \sum_{i=1}^n
  (-1)^{|x_1|+\cdots+|x_{j-1}|+j-1} x_1 \wedge \ldots \wedge \epsilon x_j \wedge \ldots \wedge x_n .
\end{equation*}
Extend $\epsilon$ to an operator
$\epsilon:\Omega_uCL_\epsilon\to\Omega_uCL_\epsilon$ of degree $1$ by the formula
\begin{equation*}
  \epsilon \< a_1 | \ldots | a_k \> = \sum_{j=1}^k (-1)^{\omega_{j-1}+1} \<
  a_1 | \ldots | \epsilon a_j \mid \ldots | a_k \> .
\end{equation*}
We have $\epsilon\circ\partial_\epsilon+\partial_\epsilon\circ\epsilon=\rho$.

Let $\partial_u=\partial/\partial u$. The operator
\begin{equation*}
  \xi_\epsilon = \xi + \epsilon \partial_u
\end{equation*}
satisfies
\begin{equation*}
  \bigl( \delta_\Omega+\delta_1+\nu_1 \bigr) \xi_\epsilon + \xi_\epsilon \bigl( \delta_\Omega+\delta_1+\nu_1 \bigr) = \rho (
  u\partial_u + 1 ) + \epsilon \partial_\epsilon - \lambda .
\end{equation*}
Form the locally finite sum
\begin{equation*}
  h_\epsilon = \sum_{j=0}^\infty (\rho(u\partial_u+1)+\epsilon \partial_\epsilon)_{j+1}^{-1} (\lambda)_j\xi_\epsilon .
\end{equation*}
Let $p_\epsilon$ be the projection from $\Omega_uCL_\epsilon$ to $SL$, identified with
the zero eigenspace of $\rho(u\partial_u+1) + \epsilon \partial_\epsilon - \lambda$. Then we have
\begin{equation*}
  \bigl( \delta_\Omega+\delta_1+\nu_1 \bigr) h_\epsilon + h_\epsilon \bigl( \delta_\Omega+\delta_1+\nu_1 \bigr) = p_\epsilon ,
\end{equation*}

The associated map $g_\epsilon:SL\to\Omega_uCL_\epsilon$ is given by the same formula as
$g:SL\to\Omega CL$, and $f_\epsilon:\Omega_uCL_\epsilon\to SL$ is given by setting
$u=\epsilon=0$ and then applying the map $f:\Omega CL\to SL$. In this way, we obtain
a contraction $(\Omega_uCL_\epsilon,SL,h_\epsilon,f_\epsilon,g_\epsilon)$.

Let $B_u\Omega_uCL_\epsilon$ be the bar construction of
$\Omega_uCL_\epsilon$ as a dg algebra over $\F[u]$: we have
\begin{equation*}
  B_u\Omega_uCL_\epsilon = \bigoplus_{k=1}^\infty \bigl( s\Omega CL_\epsilon \bigr)^{\otimes k}[u] .
\end{equation*}
The differential on $B_u\Omega_uCL_\epsilon$ is a sum
\begin{equation*}
  \delta_B + \delta_\Omega + \delta_1 + \mu + \nu = \bigl( \delta_\Omega+\delta_1+\nu_1 \bigr) + \bigl(
  \delta_B+\mu+\nu_+ \bigr) .
\end{equation*}
Since $\mu\g_\epsilon=\nu\g_\epsilon=0$, the codifferential on $BSL$ induced by the
tensor trick takes the form
\begin{multline*}
  \f_\epsilon\bigl( \delta_\Omega + \delta_1 + \nu_1 \bigr) \g_\epsilon +
  \f_\epsilon(1+(\delta_B+\mu+\nu_+)\h_\epsilon)^{-1}(\delta_B+\mu+\nu_+)\g_\epsilon \\
  = \delta_1 + \sum_{k=2}^\infty \f_\epsilon \Bigl( - (1+(\mu+\nu_+)\h_\epsilon)^{-1}\delta_B\h_\epsilon
  \Bigl)^{k-2} (1+(\mu+\nu_+)\h_\epsilon)^{-1}\delta_B\g_\epsilon .
\end{multline*}
A similar argument to Lemma~\ref{F} shows that this sum simplifies to
\begin{equation*}
  \delta_1 + \f (1+\mu\h)^{-1}\delta_B \g .
\end{equation*}
Theorem~\ref{Baranovsky} identifies the resulting \Ainf-algebra
structure on $SL$ with the universal enveloping algebra $UL$. The
\Ainf{} quasi-isomorphism from $UL$ to $\Omega_uCL_\epsilon$ that proves Lemma~19
of \cite{Tsygan} is given by the morphism of dg coalgebras
\begin{equation*}
  \h_\epsilon(1+(\delta_B+\mu+\nu_+)\h_\epsilon)^{-1}\delta_B\g_\epsilon : BUL \to B_u\Omega_uCL_\epsilon ,
\end{equation*}
or equivalently, by the twisting cochain $BUL\to\Omega_uCL_\epsilon$ given by
composition with the universal twisting cochain $B_u\Omega_uCL_\epsilon\to\Omega_uCL_\epsilon$.

\clearpage

\section*{Appendix. An alternative contracting homotopy}

In this appendix, we present the contracting homotopy of Dippell et
al. \cite{DESW}. Because this construction is based on the Koszul
complex of the graded vector space $V$, we use the counital exterior
algebra $\Wedge V_+=\Wedge V\oplus\F$. The counit
$\epsilon:\Wedge V_+\to\F$ projects from $\Wedge V_+$ to
$\Wedge_0V\cong\F$, with basis vector $1$. We will construct a homotopy
$k$ from the counital bar complex $B_+SV=BSV\oplus\F$ to $\Wedge V_+$.

If $A$ is a dg algebra and $M$ is a left dg $A$-module, the bar
construction $B_+(A,M)$ of $A$ with coefficients in $M$ has underlying
graded vector space $B_+A\otimes M$. Its differential $\ddd$ the sum of the
differentials on $BA$ and $M$ together with a single additional term
capturing the left action of $A$ on $M$:
\begin{multline*}
  \ddd[a_1|\ldots|a_k] \otimes m \\
  \begin{aligned}
    &= \sum_{i=1}^{k-1} (-1)^{\omega_i+1} \, [a_1|\ldots|a_ia_{i+1}|\ldots|a_k] \otimes m
      + (-1)^{\omega_k+1} \, [a_1|\ldots|a_{k-1}] \otimes a_km \\
    &\quad + \sum_{i=1}^k (-1)^{\omega_{i-1}} \, [a_1|\ldots|da_i|\ldots|a_k] \otimes m
      + (-1)^{\omega_k} \, [a_1|\ldots|a_k] \otimes dm .
  \end{aligned}
\end{multline*}

Consider the left $A$-module $A^+=A\oplus\F$, where the action of
$a\in A$ on $A\subset A^+$ is left multiplication, while the action of
$a\in A$ on $\F\subset A^+$ takes $1\in\F$ to $a\in A\subset A^+$. Denote the
augmentation from $A^+$ to $\F$ with kernel $A$ by
$\epsilon:A^+\to\F$. Denote the projection $a-\epsilon(a):A^+\to A$ by $\overline{a}$.

The complex $B_+(A,A^+)$ is contractible, with contracting
homotopy
\begin{equation*}
  \hhh[a_1|\ldots|a_k] \otimes a = (-1)^{\omega_k+|a|} \, [a_1|\ldots|a_k|a-\epsilon(a)] \otimes 1
  .
\end{equation*}
We have
\begin{equation*}
  (\ddd\hhh+\hhh\ddd)[a_1|\ldots|a_k] \otimes a = [a_1|\ldots|a_k] \otimes a - \delta_{k,0} \,
  [~]\otimes\epsilon(a) .
\end{equation*}

The Koszul complex of $V$ is the left $SV$-module
$SV^+\otimes\Wedge V_+$. Its differential is
\begin{equation*}
  \dd ( a \otimes b ) = (-1)^{|a|} \sum_\alpha ax^\alpha \otimes \iota_\alpha b ,
\end{equation*}
where $a\in SV^+$, $b\in \Wedge V^+$, and $\iota_\alpha$ is the linear operator of
degree $-|x^\alpha|-1$ given by the formula
\begin{equation*}
  \iota_\alpha \bigl( sx^{\alpha_1} \ldots sx^{\alpha_q} \bigr) = \sum_{i=1}^q
  (-1)^{(|x^\alpha|+1)(|x^{\alpha_1}|+\cdots+|x^{\alpha_{i-1}}|+i-1)} \, \delta^{\alpha_i}_\alpha \,
 \bigl( sx^{\alpha_1} \ldots \widehat{sx}{}^{\alpha_i} \ldots sx^{\alpha_q} \bigr) .
\end{equation*}
The Koszul complex is contractible, with contracting homotopy
\begin{equation*}
  \hh ( a \otimes b ) = \rho^{-1} \sum_\alpha (-1)^{|a|(|x^\alpha|+1)} \, \partial_\alpha a \otimes sx^\alpha b ,
\end{equation*}
where $\rho=p+q$ on $S^pV\otimes\Wedge_qV$. We have
\begin{equation*}
  (\dd\hh+\hh\dd) ( a \otimes b ) = \epsilon(a) \otimes \epsilon(b) .
\end{equation*}

The differential $\ddd$ on $B_+(SV,KV)$ equals
\begin{align*}
  \ddd [a_1|\ldots|a_k] \otimes a \otimes b
  &= \sum_{i=1}^{k-1} (-1)^{\omega_i+1} \, [a_1|\ldots|a_ia_{i+1}|\ldots|a_k] \otimes a \otimes b \\
  &+ (-1)^{\omega_k+1} \, [a_1|\ldots|a_{k-1}] \otimes a_ka \otimes b \\
  &+ \sum_{i=1}^k (-1)^{\omega_{i-1}} \, [a_1|\ldots|da_i|\ldots|a_k] \otimes a \otimes b \\
  &+ (-1)^{\omega_k} [a_1|\ldots|a_k] \otimes \Bigl( da \otimes b + (-1)^{|a|} \, a \otimes db
    \Bigr) .
\end{align*}
The homotopy
\begin{equation*}
  \hhh[a_1|\ldots|a_k] \otimes a \otimes b = (-1)^{\omega_k+|a|} \, [a_1|\ldots|a_k|a-\epsilon(a)] \otimes 1 \otimes b
\end{equation*}
on $B_+(SV,KV)$ satisfies
\begin{equation*}
  (\ddd\hhh+\hhh\ddd)[a_1|\ldots|a_k] \otimes a \otimes b = [a_1|\ldots|a_k] \otimes a \otimes b - \delta_{k0}
  [~]\otimes\epsilon(a)\otimes b .
\end{equation*}
We have morphisms of complexes
\begin{equation*}
  \begin{tikzcd}
    (\Wedge V_+,d) \rar{\ggg} & (B_+(SV,KV),\ddd) \rar{\fff} & (\Wedge
    V_+,d)
  \end{tikzcd}
\end{equation*}
given by the formulas $\ggg b = [~] \otimes 1 \otimes b$ and
\begin{equation*}
  \fff[a_1|\ldots|a_k] \otimes a \otimes b = \delta_{k0} \, \epsilon(a) \, b
\end{equation*}
such that $\ddd\hhh+\hhh\ddd = 1 - \ggg\fff$ and $\fff\ggg=1$. It is easily
checked that $\hhh^2=0$, $\fff\hhh=0$ and $\hhh\ggg=0$; thus, we obtain a
contraction
\begin{equation*}
  \begin{tikzcd}
    \hhh \,
    \rotatebox[origin=c]{270}{$\circlearrowright$} \, (B_+(SV,KV),\ddd)
    \arrow[shift left=0.25em]{r}{\fff} &
    (\Wedge V_+,d) \arrow[shift left=0.25em]{l}{\ggg}
  \end{tikzcd}
\end{equation*}

Via the isomorphism $B_+(SV,KV)\cong B_+SV \otimes KV$, the differential
$\dd$ on $KV$ induces a differential $\dd$ on $B_+(SV,KV)$,
\begin{equation*}
  \dd [a_1|\ldots|a_k] \otimes a \otimes b = (-1)^{\omega_k+|a|} \sum_\alpha [a_1|\ldots|a_k] \otimes ax^\alpha \otimes
  \iota_\alpha b .
\end{equation*}
The homotopy
\begin{equation*}
  \hh[a_1|\ldots|a_k] \otimes a \otimes b = (-1)^{\omega_k} \sum_\alpha
  (-1)^{|a|(|x^\alpha|+1)} \, [a_1|\ldots|a_k] \otimes \rho^{-1} ( \partial_\alpha a \otimes sx^\alpha b )
\end{equation*}
on $B_+(SV,KV)$ satisfies
\begin{equation*}
  (\dd\hh+\hh\dd)[a_1|\ldots|a_k] \otimes a \otimes b = [a_1|\ldots|a_k] \otimes \Bigl( a \otimes b -
  \epsilon(a) \otimes \epsilon(b) \Bigr) .
\end{equation*}
We have morphisms of complexes
\begin{equation*}
  \begin{tikzcd}
    (B_+SV,0) \rar{\gg} & (B_+(SV,KV),\dd) \rar{\ff} & (B_+SV,0)
  \end{tikzcd}
\end{equation*}
given by the formulas $\gg[a_1|\ldots|a_k] = [a_1|\ldots|a_k] \otimes 1 \otimes 1$ and
\begin{equation*}
  \ff[a_1|\ldots|a_k] \otimes a \otimes b = \epsilon(a) \epsilon(b) [a_1|\ldots|a_k] ,
\end{equation*}
such that $\dd\hh+\hh\dd = 1 - \gg\ff$ and $\ff\gg=1$. It is easily
checked that $\hh^2=0$, $\ff\hh=0$ and $\hh\gg=0$, and we obtain a
contraction
\begin{equation*}
  \begin{tikzcd}
    \hh \,
    \rotatebox[origin=c]{270}{$\circlearrowright$} \, (B_+(SV,KV),\dd)
    \arrow[shift left=0.25em]{r}{f} &
    (B_+SV,0) \arrow[shift left=0.25em]{l}{\gg}
  \end{tikzcd}
\end{equation*}

By the formula $\ddd\dd+\dd\ddd=0$, we may equally well think of $\dd$
as a perturbation of the differential $\ddd$ on $B(SV,KV)$, or of
$\ddd$ as a perturbation of the differential $\dd$.  Homological
perturbation theory applies to both perturbations, since both
$1+\dd\hhh$ and $1+\ddd\hh$ are invertible, and we obtain a pair of
contractions
\begin{equation*}
  \begin{tikzcd}[row sep=1em]
    \hhh_* \, \rotatebox[origin=c]{270}{$\circlearrowright$} \, (B_+(SV,KV),\ddd+\dd)
    \arrow[shift left=0.25em]{r}{\fff_*} &
    (\Wedge V_+,d) \arrow[shift left=0.25em]{l}{\ggg_*} \\
    \hh_* \, \rotatebox[origin=c]{270}{$\circlearrowright$} \, (B_+(SV,KV),\dd+\ddd)
    \arrow[shift left=0.25em]{r}{\ff_*} &
    (B_+SV,\delta) \arrow[shift left=0.25em]{l}{\gg_*}
  \end{tikzcd}
\end{equation*}
given by the formulas
\begin{align*}
  \hhh_* &= \hhh(1+\dd\hhh)^{-1} , &
  \fff_* &= \fff = \fff(1+\dd\hhh)^{-1} , &
  \ggg_* &= (1+\hhh\dd)^{-1}\ggg , \\
  \hh_* &= \hh(1+\ddd \hh)^{-1} , &
  \ff_* & = \ff = \ff(1+\ddd \hh)^{-1} , &
  \gg_* &= (1+\hh\ddd)^{-1}\gg .
\end{align*}
The formula $\fff\dd=0$ shows that $\fff_*=\fff$ and also that
$\fff\dd(1+\hhh\dd)^{-1}\ggg=0$, so the differential on $\Wedge V_+$ is
unperturbed. Similarly, the formula $\ff\ddd \hh=0$ shows that
$\ff\ddd(1+\hh\ddd)^{-1}\gg$ equals the differential $\delta$ on
$B_+SV$, and $\ff_*=\ff$.

The following result is due to Dippell et al. \cite{DESW} (see also
Meinrenken and Salazar \cite{MS}). Let $f_+:B_+SV\to\Wedge V_+$ and
$g_+:\Wedge V_+\to B_+SV$ be the counital versions of the morphisms
$f$ and $g$ from Section~4.
\begin{proposition}
  There is a weak contraction
  \begin{equation*}
    \begin{tikzcd}[column sep=3.5em]
      H = \ff_* \hhh_* \gg_* \,
      \rotatebox[origin=c]{270}{$\circlearrowright$} \, (B_+SV,\delta)
      \arrow[shift left=0.25em]{r}{f_+} &
      (\Wedge V_+,d) \arrow[shift left=0.25em]{l}{g_+}
    \end{tikzcd}
  \end{equation*}
\end{proposition}
\begin{proof}
  By straightforward calculation, we see that
  $f_+=\fff_* \gg_*=\fff(1+\hh\ddd)^{-1}\gg$ and
  $g_+=\ff_*\ggg_*=\ff(1+\hhh\dd)^{-1}\ggg$. We have
  \begin{align*}
    \delta H = \delta \bigl( \ff_* \hhh_* \gg_* \bigr)
    &= \ff_* (\ddd+\dd) \hhh_* \gg_* \\
    &= \ff_* ( 1 - \ggg_* \fff_*) \gg_* - \ff_* \hhh_* (\ddd + \dd) \gg_*
    = 1 - g_+f_+ - H \delta ,
  \end{align*}
  showing that $H$ gives a weak contraction.
\end{proof}

We close the appendix with an explicit formula for the homotopy
$H$. Using this formula, it may be proved that $Hg_+=0$, $f_+H=0$, and
$H^2=0$, and hence that $k$ defines a contraction from $B_+SV$ to
$\Wedge V_+$.  Since $\hh^2=0$ and $\hh\gg=0$, we have
\begin{equation*}
  \hh\ddd(-\hh\ddd)^{k-j} \gg = [\hh,\ddd](\hh\ddd)^{k-j} \gg .
\end{equation*}
It is easily checked that
\begin{multline*}
  [\hh,\ddd] [a_1|\ldots|a_k] \otimes a \otimes b \\
  = \sum_\alpha (-1)^{\omega_{k-1} + (|x^\alpha|+1)(|a_k|+|a|)} \, [a_1|\ldots|a_{k-1}] \otimes
  \rho^{-1} \bigl( (\partial_\alpha a_k)a \otimes sx^\alpha b \bigr) .
\end{multline*}
It follows that
\begin{multline*}
  (-\hh\ddd)^{k-j} \gg [a_1|\ldots|a_k] = (-1)^{\omega_k-\omega_j} \sum_{\alpha_{j+1},\ldots,\alpha_k}
  \prod_{j<p\le q\le k} \frac{(-1)^{(|x^{\alpha_p}|+1)|\partial_{\alpha_q}a_q|}}{\|a_q\|} \\
  [a_1|\ldots|a_j] \otimes (\partial_{\alpha_{j+1}}a_{j+1})\ldots(\partial_{\alpha_k}a_k) \otimes sx^{\alpha_{j+1}} \ldots
  sx^{\alpha_k} .
\end{multline*}
Let $\sgn(\pi)$ be the sign \eqref{sgn} determined by the Koszul sign
convention for the action of the symmetric group on $\Wedge_{k-j}V$.
We conclude that the homotopy $k$ is given by the explicit formula
\begin{align*}
  H &= \sum_{j=0}^k \ff \hhh(\dd\hhh)^{k-j} (\hh\ddd)^{k-j} \gg \\
    &= \sum_{j=0}^k (-1)^{\omega_j} \sum_{\alpha_{j+1},\ldots,\alpha_k} \prod_{j<p\le q\le k}
      \frac{ (-1)^{(|x^{\alpha_p}|+1)|\partial_{\alpha_q}a_q|} }{\|a_q\|}
      \sum_{\pi\in S_{k-j}} \sgn(\pi) \\
    & \qquad [a_1|\ldots|a_j| A_j |x^{\alpha_{j+\pi(1)}} | \ldots | x^{\alpha_{j+\pi(k-j)}} ] ,
\end{align*}
where
$A_j=(\partial_{\alpha_{j+1}}a_{j+1})\ldots(\partial_{\alpha_k}a_k) -
\epsilon\bigl((\partial_{\alpha_{j+1}}a_{j+1})\ldots(\partial_{\alpha_k}a_k)\bigr)$.

\begin{ack}
  We thank Alex Karapetyan and Boris Tsygan for sharing their insights
  on operations acting on the bar construction, and Jim Stasheff,
  Bruno Vallette and the referee for helpful comments on earlier
  versions.
\end{ack}

\end{document}